\documentclass[12pt]{amsart}

\usepackage{graphicx,latexsym,amsfonts,amssymb,txfonts,amsmath,amsthm}
\usepackage{pdfsync,color,tabularx,rotating}
\usepackage[all,cmtip]{xy}
\usepackage{url}
\usepackage{tikz}
\usepackage{epsfig}
\usepackage{comment}

\advance\hoffset by -0.9 truecm

\newtheorem{theo}{Theorem}[section]

\newtheorem{lemm}[theo]{Lemma}
\newtheorem{conj}[theo]{Conjecture}

\theoremstyle{remark}
\newtheorem{rema}[theo]{\bf Remark}

\theoremstyle{remark}
\newtheorem{exam}[theo]{\bf Example}

\theoremstyle{remark}

\theoremstyle{remark}

\theoremstyle{remark}

\begin{document}

\title{Groups of prime degree and the Bateman--Horn Conjecture}

\author{Gareth A. Jones and Alexander K. Zvonkin}

\address{School of Mathematical Sciences, University of Southampton, Southampton SO17 1BJ, UK}
\email{G.A.Jones@maths.soton.ac.uk}

\address{LaBRI, Universit\'e de Bordeaux, 351 Cours de la Lib\'eration, F-33405 Talence 
Cedex, France}
\email{zvonkin@labri.fr}

\subjclass[2010]{11A41, 11N05, 11N32, 20B05, 20B25, 20H20}.

\keywords{Permutation group, linear group, prime degree, projective space, 
Bunyakovsky conjecture, Bateman--Horn conjecture, Goormaghtigh conjecture.}

\begin{abstract}
As a consequence of the classification of finite simple groups, the classification of 
permutation groups of prime degree is complete, apart from the question of when the 
natural degree $(q^n-1)/(q-1)$ of ${\rm PSL}_n(q)$ is prime. We present heuristic arguments 
and computational evidence based on the Bateman--Horn Conjecture
to support a conjecture that for each prime $n\ge 3$ there are 
infinitely many primes of this form, even if one restricts to prime values of $q$.
Similar arguments and results apply to the parameters of the simple groups
${\rm PSL}_n(q)$, ${\rm PSU}_n(q)$ and ${\rm PSp}_{2n}(q)$ which arise in the work of Dixon and Zalesskii on linear groups of prime degree.
\end{abstract}

\maketitle

%%%%%%%%%%%%%%%%%%

\section{Permutation groups of prime degree}\label{sec:prime}

One of the oldest problems in Group Theory is to classify the permutation groups of prime degree, originally studied in terms of the solution of polynomial equations of prime degree. Let $G$ be a transitive permutation group of prime degree $p$. In 1831 Galois~\cite{Galois} proved that $G$ is solvable if and only if $G$ is (isomorphic to) a subgroup of the $1$-dimensional affine group
\[{\rm AGL}_1(p)=\{t\mapsto at+b\mid a, b\in{\mathbb F}_p, a\ne 0\}\cong {\rm C}_p\rtimes {\rm C}_{p-1}\]
containing the translation subgroup $\{t\mapsto t+b\}\cong {\rm C}_p$. There is one such group $G$ for each $d$ dividing $p-1$, namely
\[\{t\mapsto at+b\mid a, b\in{\mathbb F}_p, a^d=1\}\cong {\rm C}_p\rtimes {\rm C}_d.\]

In 1906 Burnside (\cite{Bur06}, \cite[\S 251]{Bur11}) proved that if $G$ is nonsolvable then $G$ is $2$-transitive.
In this case $G$ has a unique minimal normal subgroup $S\ne 1$ which is simple and also 
$2$-transitive, with centraliser $C_G(S)=1$, so that $G\le{\rm Aut}\,S$. This reduces the problem to studying nonabelian simple groups $S$ of degree $p$ and their automorphism groups. The classification of finite simple groups (announced around 1980) implies a classification of those with 2-transitive actions (see~\cite{Cam81} or~\cite{Feit}, for example). Most of these have composite degree; those of prime degree are as follows:

\begin{itemize}
\item[a)] $S={\rm A}_p$, $G={\rm S}_p$, for primes $p\ge 5$;
\item[b)] $S={\rm PSL}_n(q)\le G\le {\rm P\Gamma L}_n(q)={\rm PGL}_n(q)\rtimes{\rm Gal}\,{\mathbb F}_q$ in cases where the natural degree $m:=(q^n-1)/(q-1)$ of these groups is prime;
\item[c)] $S={\rm PSL}_2(11)$, ${\rm M}_{11}$ and ${\rm M}_{23}$ for $p=11, 11$ and $23$.
\end{itemize}
 In (b) the groups act on the $m$ points (or $m$ points and hyperplanes if $n\ge 3$) of the projective geometry ${\mathbb P}^{n-1}({\mathbb F}_q)$ for a prime power $q$. In (c), ${\rm PSL}_2(11)$ acts on the $11$ cosets of a subgroup $H\cong {\rm A}_5$ (two conjugacy classes, giving two actions, equivalent to those on the vertices and cells of the hendecachoron or 11-cell, a nonorientable 4-polytope discovered independently by Gr\"unbaum~\cite{Gru} and Coxeter~\cite{Cox}; see also \cite{11-cell}); ${\rm M}_{11}$ and ${\rm M}_{23}$ are Mathieu groups, acting on block designs with $11$ and $23$ points.

Unfortunately, this result does not tell us when the degree $m$ in (b) is prime. Indeed, it is unknown whether there are finitely or infinitely many such `projective primes', as we will call them.

\medskip

\noindent{\bf Open Problem:} In (b), is the degree
\[m=\frac{q^n-1}{q-1}=1+q+q^2+\cdots+q^{n-1}\quad (q\;\;\hbox{a prime power})\]
prime in finitely or infinitely many cases?

\medskip

If $n=2$ the projective primes $m$ are the Fermat primes $1+2^e$, $e=2^f$;
the only known examples are  $3, 5, 17, 257, 65537$ for $f\le 4$.
If $q=2$ the primes $m$ are the Mersenne primes $2^n-1$, $n$ prime;
at the time of writing, $51$ examples $3, 7, 31, \ldots, 2^{82\,589\,933}-1$ are known.
It is widely conjectured that there are no further Fermat primes, but infinitely many Mersenne primes.
These are very old and difficult problems; with nothing new to say about them,
we will assume from now on that $n, q\ge 3$.

Our main conjecture is that {\em there are infinitely many projective primes}.
The goal of this note is to present heuristic arguments and computational evidence to support this conjecture.
See~\cite{JZ:primes} for further details, \cite{JZ:Klein} for applications to dessins d'enfants (maps on surfaces representing curves defined over algebraic number fields), and~\cite{JZ:block} for a similar problem involving block designs.

%%%%%%%%%%%%%%%%%

\section{Conjectures}\label{sec:conj}

If a polynomial $f(t)\in{\mathbb Z}[t]$ takes infinitely many prime values for $t\in{\mathbb N}$ then clearly
\begin{itemize}
\item its leading coefficient is positive,
\item it is irreducible in ${\mathbb Z}[t]$, and
\item it is not identically zero modulo any prime.
\end{itemize}
In 1857 Bunyakovsky, the discoverer of the infinite-dimensional form of the Cauchy--Schwarz inequality, conjectured in~\cite{Bun-1857} that these conditions are also sufficient. (The last condition is needed to exclude cases like $t^2+t+2$, which is irreducible but takes only even values.) The case $\deg f=1$ is true: this is Dirichlet's Theorem on primes in an arithmetic progression (see~\cite[\S5.3]{BS}). No other case is proved, not even $t^2+1$, studied by Euler~\cite{Eul} and Landau.
Writing $q=p^e$ we require the result for $f(t)=1+t^e+t^{2e}+\cdots+t^{(n-1)e}$, but with the extra condition that $t$ should also be prime.

Schinzel's Hypothesis H~\cite{SS} deals with this restriction by extending Bunyakovsky's 
conjecture to finite {\em sets\/} of polynomials $f_1,\ldots,f_k$ simultaneously taking 
prime values infinitely often.
An obvious necessary condition is that each
$f_i$ should satisfy the first two Bunyakovsky conditions, while the third is that
$f(t):=\prod_{i=1}^k f_i(t)$ should not be identically zero modulo any prime.
For example, $t(t+1)$ is identically zero mod~$(2)$, while $t(t+2)$ is not.
It is conjectured that these conditions are also sufficient,
but as in the case of the Bunyakovsky Conjecture this has been proved only in the case 
$k=1$, $\deg f_1=1$.
(However, see~\cite{SkSo} for recent evidence in support of Hypothesis H.)

In addition to cases with $k=1$, such as the Euler--Landau problem,
conjectures which would follow from a proof of Hypothesis H include
\begin{itemize}
\item $f_1=t$, $f_2=t+2$, the twin primes conjecture;
\item $f_1=t$, $f_2=2t+1$, the Sophie Germain primes conjecture;
\item $f_1=t$, $f_2=1+t^{e}+t^{2e}+\cdots+t^{(n-1)e}$ for fixed $e$ and $n$, particular cases of our projective primes conjecture, provided $f_2$ is irreducible (see Section~\ref{sec:irred}).
\end{itemize}

\begin{rema}[Hypothesis ${\rm H}_0$]\label{rem:H_0}
In the same paper \cite{SS}, the authors formulate an apparently weaker conjecture 
${\rm H}_0$: 
under the same conditions as above, the values $f_1(t),\ldots,f_k(t)$ are all prime for 
{\em at least one positive integer}\/ $t$. It turns out, however, that ${\rm H}_0$ implies 
$\rm H$, and this fact is trivial! It suffices to consider the sets of polynomials
$f_1(t+c),\ldots,f_k(t+c)$ for constants $c\in\mathbb{N}$, and to note that, 
according to ${\rm H}_0$, the values of the polynomials in each of the sets are all prime 
for at least one integer $t>0$.
\end{rema}

In 1962 Bateman and Horn~\cite{BH} proposed a quantified version of Schinzel's Hypothesis H which, if proved, would imply all the above conjectures (see~\cite{AFG} for an excellent survey). 

\begin{conj}[The Bateman--Horn Conjecture (BHC)]
If distinct polynomials $f_1,\ldots,f_k$ satisfy the above conditions, and $Q(x)$ is the 
number of positive integers $t\le x$ such that $f_1(t),\ldots, f_k(t)$ are all prime, then
\begin{equation}\label{eq:BH-Q}
Q(x)\sim E(x):=\frac{C}{\prod_{i=1}^k\deg f_i}\int_2^x\frac{dt}{(\ln t)^k}
\quad\hbox{as}\quad x\to\infty,
\end{equation}
where
\begin{equation}\label{eq:BH-C}
C=C(f_1,\ldots, f_k):=\prod_{{\rm prime}\,r}\left(1-\frac{1}{r}\right)^{-k}\left(1-\frac{\omega_f(r)}{r}\right)
\end{equation}
with the product over all primes $r$, and $\omega_f(r)$ is the number of solutions in 
${\mathbb F}_r$ of $f(t)=0$. 
\end{conj}

The infinite product converges to a limit $C>0$ (see~\cite{AFG} for a proof), and $\int_2^{\infty}dt/(\ln t)^k$ diverges for each $k\ge 1$, so $E(x)\to\infty$ with $x$; thus $f_1(t),\ldots,f_k(t)$ are simultaneously prime for infinitely many~$t$ provided the conjecture is true. However, it is proved only in the case of Dirichlet's Theorem. Since
\[\int_2^x\frac{dt}{(\ln t)^k}=\frac{x}{(\ln x)^k}+O\left(\frac{x}{(\ln x)^{k+1}}\right),\]
there is an alternative form
\begin{equation}\label{eq:BH-Q'}
Q(x)\sim H(x):=\frac{C}{\prod_{i=1}^k\deg f_i}\cdot\frac{x}{(\ln x)^k}
\quad\hbox{as}\quad x\to\infty
\end{equation}
for the estimate, which can be more convenient but significantly less accurate. 

\begin{exam}
Taking $k=1$ and $f_1=f=t$ we get $\omega_f(r)=1$ for all prime $r$, so that $C=1$.
Therefore, we obtain the two familiar versions of the Prime Number Theorem:
\[\pi(x)\sim {\rm Li}(x):=\int_2^x\frac{dt}{\ln t}\sim\frac{x}{\ln x}.\]
The function ${\rm Li}(x)$ is also called the {\em offset logarithmic integral function}.
The estimate $x/\!\ln x$ is that of Hadamard and de la Vall\'ee Poussin, while the estimate
${\rm Li}(x)$ is a particular case of the BHC. To compare these two estimates, let us take
$\pi(10^{25})=176\,846\,309\,399\,143\,769\,411\,680$ (see the entry A006880 of \cite{OEIS}).
Then the relative error of the estimate $10^{25}\!/\!\ln(10^{25})$ is $-1.77\,\%$, 
while that of the estimate ${\rm Li}(10^{25})$ is $3.12\cdot 10^{-11}\,\%$.
\end{exam}

For a heuristic proof of the BHC see the original paper \cite{BH} by Bateman and Horn
and also a recent overview \cite{AFG}.

\begin{rema}[An improved estimate]\label{rem:Li}
Li~\cite{Li} has recently proposed a modification of the BHC, in which 
$1/\!\ln f_i(t)$ is used instead of $1/d_i\ln t$. This gives significantly better 
estimates $E(x)$ in cases such as the Sophie Germain primes conjecture involving 
a non-monic polynomial $f_i$, but when each $f_i$ is monic, as in our case, 
the effect is negligible. 
\end{rema}

%%%%%%%%%%%%%%%%%%%%%%

\section{Irreducibility of the polynomial $(t^{ne}-1)/(t^e-1)$}\label{sec:irred}

In order to apply the BHC to the projective groups of prime degree we consider
two polynomials, $f_1=t$ and $f_2=(t^{ne}-1)/(t^e-1)$, and we need to ensure that 
the polynomial $f_2$ is irreducible.

\begin{lemm}\label{le:irred}
Given integers $n\ge 2$ and $e\ge 1$, the polynomial
\[f_2(t)=\frac{t^{ne}-1}{t^e-1}=1+t^e+t^{2e}+\cdots+t^{(n-1)e}\]
is irreducible in ${\mathbb Z}[t]$ if and only if $n$ is prime and $e$ is a power $n^i\;(i\ge 0)$ of $n$.
\end{lemm}

\noindent{\sl Proof}.
If $k\in\mathbb{N}$ the {\em cyclotomic polynomial}\/ $\Phi_k(x)$ is, by definition,
the polynomial with integer coefficients whose roots are the primitive $k$th roots of unity.
It is irreducible and has degree $\varphi(k)$, where $\varphi$ is the Euler totient function.
For any $n\in\mathbb{N}$ we have $x^n-1=\prod_{d|n}\Phi_d(x)$
(see~\cite[\S 5.2.1]{BS} or \cite[\S 4.3, Problem~26]{NZM}). 
Putting $x=t^e$ gives
\begin{equation}\label{eq:cyclotomic}
f_2(t) = \frac{t^{ne}-1}{t^e-1} = \prod_d \Phi_d(t),
\end{equation}
with the product over all $d$ which divide $ne$ but not $e$. 
Thus $f_2$ is irreducible if and only if there is just one such divisor $d$
(which is $ne$ itself, of course). By considering the prime power decompositions of $e$ and $ne$
one can see that this happens if and only if $n$ is prime and $e$ is a power of $n$. \hfill$\Box$

%%%%%%%%%%%%%%%%%%%%%%

\section{Primality testing}

To find $Q(x)$ for various large $x$, we used the Rabin--Miller (RM) primality test~\cite{Rabin}.
It determines whether a given number is prime or composite without trying to 
factor it but by checking independent instances of a necessary primality condition. There is a real abyss 
between the complexities of the most efficient factoring algorithms and the RM-test. 
To give but one example, it took 4400 GHz-years to factor a 232-digit number into 
two 116-digit primes, see~\cite{RSA}. The RM-test gives a correct answer 
(``the number is composite'') in less than 0.0005 seconds on a very modest laptop.

The RM-test is probabilistic. If it affirms that a given number is composite, then
it is indeed composite. If, however, the test affirms that a number is prime, 
the number may turn out to be composite. The probability 
of such an event is infinitesimally small: during 40 years of widespread use 
of the RM-test not a single such error has ever been reported\footnote{A dialogue from Gilbert 
and Sullivan's {\em I am the Captain of the Pinafore}\/ comes to mind: ``What, never? 
No, never. What, never? Well, hardly ever''.}. 
Note also that long computations are prone to hardware errors.
If, however, by incredibly bad luck a few of our `primes' are composite, this would not 
invalidate our evidence of literally millions of projective primes.

%%%%%%%%%%%%%%%%%%%%%%

\section{Applying the Bateman--Horn Conjecture to projective groups}
\label{sec:applying}

\subsection{Relative abundance of types of projective primes}

We tested the BHC estimates for projective primes against the results of computer searches. Define a projective prime $m=1+q+\cdots+q^{n-1}$ with $q=p^e$, $p$ prime, to have {\em type} $(e,n)$. For each type satisfying Lemma~\ref{le:irred} define $P(x)=P_{(e,n)}(x)$ to be the number of primes $p\le x$ such that $p$ and $m$ are prime, and let $E(x)=E_{(e,n)}(x)$ be the corresponding Bateman--Horn estimate~(\ref{eq:BH-Q}) for $P_{(e,n)}(x)$, formed using the polynomials
\[f_1(t)=t\quad\hbox{and}\quad f_2(t)=1+t^e+t^{2e}+\cdots+t^{(n-1)e}.\]

The smallest projective primes $m$, as a function of $p$, are those of type $(1,3)$, of the form $m=1+p+p^2$ with $p$ prime (recall that we have excluded the case $n=2$), so 
this type appears most frequently in searches up to a given bound. For example, all but $301$ of the $1\,974\,311$ projective primes $m\le 10^{18}$ have type $(1,3)$. The second most frequent type is $(1,5)$, with $252$ examples $m\le 10^{18}$.

\begin{table}[htbp]
\begin{center}
\begin{tabular}{l|c|c|c|c}
\hspace*{11mm}Segment & \#(prime $p$) &
\#(prime $m$) & ratio & $\max p$ \\
\hline
\hspace*{9mm} $2,\ldots,\, 10^{10}$ 			  & 455\,052\,511    &
15\,801\,827 & 3.473\% & 9\,999\,999\,491  \\
\hspace*{4mm} $10^{10},\ldots,\, 2\cdot 10^{10}$  & 427\,154\,205    & 
13\,882\,936 & 3.250\% & 19\,999\,999\,757 \\
$2\cdot 10^{10},\ldots,\, 3\cdot 10^{10}$ 		  & 417\,799\,210    & 
13\,279\,095 & 3.178\% & 29\,999\,999\,921 \\
$3\cdot 10^{10},\ldots,\, 4\cdot 10^{10}$ 		  & 411\,949\,507    & 
12\,913\,713 & 3.135\% & 39\,999\,999\,719 \\
$4\cdot 10^{10},\ldots,\, 5\cdot 10^{10}$ 		  & 407\,699\,145    & 
12\,645\,233 & 3.102\% & 49\,999\,999\,619 \\
$5\cdot 10^{10},\ldots,\, 6\cdot 10^{10}$ 		  & 404\,383\,577    & 
12\,439\,618 & 3.076\% & 59\,999\,999\,429 \\
$6\cdot 10^{10},\ldots,\, 7\cdot 10^{10}$		  &	401\,661\,384    &
12\,274\,191 & 3.056\% & 69\,999\,999\,287 \\
$7\cdot 10^{10},\ldots,\, 8\cdot 10^{10}$		  &	399\,359\,707 	 &
12\,136\,112 & 3.039\% & 79\,999\,999\,679 \\
$8\cdot 10^{10},\ldots,\, 9\cdot 10^{10}$		  & 397\,369\,745	 &
12\,010\,780 & 3.023\% & 89\,999\,999\,981 \\
$9\cdot 10^{10},\ldots,\, 10^{11}$  			  & 395\,625\,822	 &
11\,910\,803 & 3.011\% & 99\,999\,999\,977 \\
\hline
\hspace*{13mm}Total 	  						  & 4\,118\,054\,813 &
129\,294\,308 & 3.140\% & 99\,999\,999\,977
\end{tabular}
\end{center}
\vspace{2mm}
\caption{The second column gives the number of primes in the corresponding
segment, while the third column gives the number of those primes $p$ which
yield a projective prime $m=1+p+p^2$. The proportion of such primes among all the
primes of the second column is given in the fourth column.}
\label{tab:betrema}
\end{table}

As further evidence for the abundance of projective primes of type $(1,3)$, our colleague 
Jean B\'etr\'ema examined all primes $p\le 10^{11}$ using the package {\tt Primes.jl} 
of the language {\tt Julia}.
This is much more efficient than Maple for problems of this sort. We partially reproduce 
B\'etr\'ema's results in Table \ref{tab:betrema}. 

Thus 129\,294\,308 primes $p\le 10^{11}$ give a  prime $m$ of type $(1,3)$;
the largest is 99\,999\,999\,977, with $m=9\,999\,999\,995\,500\,000\,000\,507$.
The ratio decreases as the upper limit grows, but it seems reasonable
to conjecture that even in this restricted case there are infinitely many projective primes.

In making our estimates, we concentrated on the apparently most abundant case of type $(1,3)$, though we did not neglect other apparently less frequent types, such as $(1,5)$ and $(3,3)$. For types $(1,n)$ with $n$ prime the polynomials $f_1=t$ and $f_2=1+t+t^2+\cdots+t^{n-1}$ satisfy the conditions of the BHC. The roots of $f=f_1f_2$ in ${\mathbb F}_r$ are $0$ for all primes $r$, together with $1$ if $r=n$, and the $n-1$ primitive $n$-th roots of $1$ if $r\equiv 1$ mod~$(n)$, so $\omega_f(r)=2$, $n$ or $1$ as $r=n$, $r\equiv 1$ mod~$(n)$ or otherwise.

%%%%%%%%%%%%

\subsection{Type $(1,3)$.} Using these values for $n=3$, we computed 
$C=C(f_1,f_2)=1.521730$ 
by taking partial products in~(\ref{eq:BH-C}) over the primes $r\le 10^9$. To count primes $m=1+p+p^2\le 10^{18}$ we took $p\le x=10^9$ (solving $1+x+x^2=10^{18}$ would be more precise, but the difference is negligible). Using numerical integration, Maple gives
\[\int_2^x\negthinspace\frac{dt}{(\ln t)^2} = 2\,594\,294.364,\]
leading to an estimate
\[E(x)=E_{(1,3)}(x)=\frac{C}{2}\int_2^x\negthinspace\frac{dt}{(\ln t)^2} =
1\,973\,907.86.\]
Comparing this with the true value $P(x)=P_{(1,3)}(x)=1\,974\,010$, found by computer search, shows that the error in $E(x)$ is about $-0.0052\,\%$.

As a second experiment with type $(1,3)$ we took $x=i\cdot 10^{10}$ for $i=1, 2, \ldots, 10$.
Table~\ref{tab:n=3BHratios} gives the resulting values of $P(x)$, $E(x)$ and $E(x)/P(x)$. The maximum relative error, attained in the first line, is $0.034\,\%$.

\begin{table}[htbp]
\begin{center}
\begin{tabular}{c|c|c|c}
$x$ & $P(x)$ & $E(x)$ & $E(x)/P(x)$ \\
\hline
$1 \cdot 10^{10}$ &  15\,801\,827 & $1.579642126 \times 10^7$  & 0.9996579044 \\
$2 \cdot 10^{10}$ &  29\,684\,763 & $2.968054227 \times 10^7$  & 0.9998578150 \\
$3 \cdot 10^{10}$ &  42\,963\,858 & $4.296235691 \times 10^7$  & 0.9999650617 \\
$4 \cdot 10^{10}$ &  55\,877\,571 & $5.587447496 \times 10^7$  & 0.9999445924 \\
$5 \cdot 10^{10}$ &  68\,522\,804 & $6.852175590 \times 10^7$  & 0.9999847043 \\
$6 \cdot 10^{10}$ &  80\,962\,422 & $8.096382889 \times 10^7$  & 1.0000173771 \\
$7 \cdot 10^{10}$ &  93\,236\,613 & $9.323905289 \times 10^7$  & 1.0000261688 \\
$8 \cdot 10^{10}$ & 105\,372\,725 & $1.053741048 \times 10^8$  & 1.0000130940 \\
$9 \cdot 10^{10}$ & 117\,383\,505 & $1.173885689 \times 10^8$  & 1.0000431394 \\
$10^{11}$         & 129\,294\,308 & $1.292974079 \times 10^8$  & 1.0000239757 
\end{tabular}
\end{center}
\vspace{2mm}
\caption{The second column gives the numbers $P(x)=P_{1,3}(x)$ of projective primes 
$m=1+p+p^2$ for primes $p\le x=i\cdot 10^{10}$, where $i=1,\ldots,10$ (the cumulative 
totals from Table~\ref{tab:betrema}),
the third column gives the corresponding Bateman--Horn estimates $E(x)=E_{1,3}(x)$ for $P(x)$,
and the fourth column gives the ratios $E(x)/P(x)$.}
\label{tab:n=3BHratios}
\end{table}

%%%%%%%%%%%

\subsection{Type $(1,5)$} For projective primes of type $(1,5)$, using $f_1=t$ and $f_2=1+t+\cdots+t^4$ we found that $C=2.571048$. To count such primes 
$m\le 10^{18}$ we took $x=10^{9/2}$. Maple gives
\[\int_2^x\negthinspace\frac{dt}{(\ln t)^2}=383.84,\]
so that
\[E_{(1,5)}(x)=\frac{C}{4}\int_2^x\negthinspace\frac{dt}{(\ln t)^2}=246.72,\]
compared with the true value $P_{(1,5)}(x)=252$.

%%%%%%%%%%%%%%

\subsection{Type $(3,3)$} With $f_1=t$ and $f_2=1+t^3+t^6$,
we found that $C=2.086089$. Taking $x=10^3$, Maple gives
\[E_{3,3}(x)=\frac{C}{6}\int_2^x\negthinspace\frac{dt}{(\ln t)^2}=12.06,\]
compared with the true value $P_{(3,3)}(x)=10$.

%%%%%%%%%%%%%%%%%

\subsection{Other types $(e,n)$} 

For other fixed types $(e,n)$ there are too few projective primes within our range of feasible computation for comparisons to be meaningful. Nevertheless, in all cases $E_{(e,n)}(x)\to\infty$ as $x\to\infty$, so the accuracy of the above estimates encourages us to conjecture that there are infinitely many projective primes of each possible type $(e,n)$.

%%%%%%%%%%%%%%%%

\subsection{Fixed $q$, $n\to\infty$}

Computer searches for fixed $q$ and $n\to\infty$ are even more difficult, and the BHC no longer applies (though similar heuristic estimates are possible), so rather than making a conjecture we simply ask whether any fixed $q$ (necessarily prime, by Lemma~\ref{le:irred}) yields infinitely many projective primes. This generalises the Mersenne primes problem for $q=2$.

%%%%%%%%%%%%%%%

\section{Groups of prime power degree}

Although Section~\ref{sec:applying} of this paper concentrates on those 
cases where the natural degree $m$ of ${\rm PSL}_n(q)$ is prime, 
there is also interest in cases such as ${\rm PSL}_2(8)$ and ${\rm PSL}_5(3)$ where $m$ is 
a prime power ($3^2$ and $11^2$ respectively). For instance, Guralnick~\cite{Gur} has shown that 
if a nonabelian simple group $S$ has a transitive representation of prime power degree, 
then $S$ is an alternating group or ${\rm PSL}_n(q)$ acting naturally, or ${\rm PSL}_2(11)$, 
${\rm M}_{11}$ or ${\rm M}_{23}$ acting as in (c) in Section~\ref{sec:prime}, or the unitary 
group ${\rm U}_4(2)\cong {\rm Sp}_4(3)\cong {\rm O}_5(3)$ permuting the $27$ lines on 
a cubic surface. In particular, $S$ is doubly transitive in all cases except the last, 
where it has rank~3. See also~\cite{EGSS},
where Estes, Guralnick, Schacher and Straus have shown that for each prime $p$
there are only finitely many $e, q, n\ge 3$ such that $p^e=(q^n-1)/(q-1)$.

If $n$ is composite then ${\rm PSL}_n(q)$ cannot have prime degree, but could it have prime power degree?
More generally, while a reducible polynomial $f(t)\in{\mathbb Z}[t]$ can take only finitely many prime values,
can it take infinitely many prime power values? This issue is addressed in~\cite{JZ:block}.

%%%%%%%%%%%%%%%

\section{Linear groups of prime degree}\label{sec:linear}

One can also apply this technique to other situations within Group Theory, such as the classification of linear groups of prime degree, where `degree' in this context means the degree, or dimension $m$, of a faithful irreducible matrix representation over $\mathbb C$. For example, in~\cite{DZ98} Dixon and Zalesskii have classified the finite primitive subgroups $G\le{\rm SL}_m({\mathbb C})$, for prime $m$, that is, those which preserve no non-trivial direct sum decomposition of the natural module ${\mathbb C}^m$. The centre $Z$ of $G$, consisting of scalar matrices, has order $1$ or $m$; if the socle (subgroup generated by the minimal normal subgroups) $M$ of $G/Z$ is abelian then $G/Z$ is an extension of a normal subgroup $M\cong{\rm C}_m\times{\rm C}_m$ by an irreducible subgroup of ${\rm SL}_2(m)$, all of which are known; the authors therefore concentrate on the case where $M$ is non-abelian, dealing in the main paper with the case where $M$ acts primitively, and in a corrigendum with the imprimitive case (see Subsection~\ref{subsec:imprim} for the latter).

If $M$ is primitive then it is a non-abelian simple group $S$ with $G/Z\le{\rm Aut}\,S$. Theorem~1.2 of~\cite{DZ98} gives a finite list of families of simple groups $S$ which can arise, with necessary and sufficient conditions on $m$ and their parameters for such groups $G$ to exist. This result is analogous to our description in Section~\ref{sec:prime} of the permutation groups of prime degree, in the sense that for some families it is unknown whether these conditions are satisfied by finitely or infinitely many sets of parameters. For several of these families one can provide evidence for the latter by using the BHC in the same way as we have applied it to permutation groups ${\rm PSL}_n(q)$ of prime degree. The relevant cases are as follows.

%%%%%%%%%%%

\subsection{Unitary groups} As a simple example, Case~(4) of Theorem~1.2 includes 
groups $G$ for which $S$ is isomorphic to the unitary group 
${\rm PSU}_n(q)$, where the degree
\[m=\frac{q^n+1}{q+1}=1-q+q^2-\cdots+q^{n-1}\]
of the representation is prime, so that $n$ is an odd prime. (Here, as usual, $q$ denotes a prime power.) It is unknown whether there are finitely or infinitely many such pairs $(n,q)$ for which $m$ is prime.

The BHC estimates $E(x)$ for the pair of irreducible polynomials
\begin{equation}\label{eq:PSU}
f_1(t)=t\quad\hbox{and}\quad f_2(t)=1-t^e+t^{2e}-\cdots+t^{(n-1)e}
\end{equation}
are identical to those for $f_1(t)=t$ and $f_2(t)=1+t^e+t^{2e}+\cdots+t^{(n-1)e}$ which we found in Section~\ref{sec:applying}: the values of $\omega_f(r)$ are the same for all primes $r$, since there is a bijection $t\mapsto-t$ between the roots of the two polynomials $f=f_1f_2$ mod~$(r)$ for each $r$, while all other ingredients of (\ref{eq:BH-Q}) and (\ref{eq:BH-C}) are unchanged. It follows that our earlier estimates $E_{(e,n)}(x)$ for permutation groups ${\rm PSL}_n(q)$ of degree $(q^n-1)/(q-1)$ all apply in this new situation. The only difference is in the verification of these estimates, where we determine the actual number $Q(x)=Q_{(e,n)}(x)$ of primes $t\le x$ such that $1-t^e+t^{2e}-\cdots+t^{(n-1)e}$ is prime.

Some results of this kind are shown in Table~\ref{tab:up-to-10^{18}}, where the first column shows the type $(e,n)$, and the second and third columns show the numbers of primes $m\le 10^{18}$ of the forms $(q^n-1)/(q-1)$ and $(q^n+1)/(q+1)$, where $q=p^e$. (The prime $m=31$ is counted twice in the second column, once each for types $(1,3)$ and $(1,5)$.) The exceptional cases are defined to be those not of type $(1,3)$.

Since the BHC estimates $E_{(e,n)}(x)$ for these two families of primes are identical, and are almost identical with Li's amendment, and since they agree very closely and fairly closely with the computer searches in the two main cases of types $(1,3)$ and $(1,5)$, we extend our conjecture of infinitely many primes $(q^n-1)/(q-1)$ for any given prime $n\ge 3$ to those of the form $(q^n+1)/(q+1)$, and hence to the associated linear groups $G$ of these degrees.

\medskip

\vspace{1cm}

\begin{table}[htbp]
\begin{center}
\begin{tabular}{c|c|c}
$(e,n)$ & $(q^n-1)/(q-1)$ & $(q^n+1)/(q+1)$ \\
\hline\hline
$(1,3)$  & $1\,974\,010$   & $1\,973\,762$ \\
\hline\hline
$(1,2)$  & $1$	 & -- \\
$(2,2)$  & $1$	 & -- \\
$(4,2)$  & $1$	 & -- \\
$(8,2)$  & $1$	 & -- \\
$(16,2)$ & $1$	 & -- \\
\hline
$(1,5)$  & $252$ & $232$ \\
$(1,7)$  & $21$  & $24$ \\
$(1,11)$ & $3$	 & $3$ \\
$(1,13)$ & $4$   & $3$ \\
$(1,17)$ & $2$	 & $3$ \\
$(1,19)$ & $1$	 & $2$ \\
$(1,23)$ & --	 & $2$ \\
$(1,31)$ & $1$	 & $1$ \\
$(1,43)$ & --	 & $1$ \\
$(1,61)$ & --	 & $1$ \\
\hline
$(3,3)$	 & $10$	 & $9$ \\
$(5,5)$  & --	 & $1$ \\
$(7,7)$  & $1$	 & $1$ \\
$(9,3)$  & $1$	 & $1$ \\
\hline\hline
{\sc Total}	& $1\,974\,311$ & $1\,974\,046$ \\
exceptional & $301$ & $284$
\end{tabular}
\end{center}
\caption{The second and third columns show the numbers of primes $m\le 10^{18}$ 
of the forms $(q^n-1)/(q-1)$ and $(q^n+1)/(q+1)$, where $q=p^e$. The type $(e,n)$ is indicated
in the first column.}
\label{tab:up-to-10^{18}}
\end{table}

%%%%%%%%%

\subsection{Projective special linear groups} Several other cases in~\cite[Theorem~1.2]{DZ98} can be treated in a similar way by using the BHC. For example, Case~(2)(ii) includes groups $G$ with $S\cong{\rm PSL}_2(q)$ where $q$ and the degree $m=(q-1)/2$ are both prime, that is, $m$ is a Sophie Germain prime, one for which $2m+1$ is also prime. In this case we can take
\begin{equation}\label{eq:PSL2}
f_1(t)=t\;(=m) \quad \mbox{and} \quad f_2(t)=2t+1\;(=q),
\end{equation}
giving
\[C=C(f_1,f_2)=2\negthickspace\negthickspace\prod_{{\rm prime}\,r>2}\left(1-\frac{1}{r}\right)^{-2}\left(1-\frac{2}{r}\right)\]
%=1.320324\]
(the same as for twin primes, where $f_i(t)=t$ and $t+2$, since in both cases $\omega_f(r)=1$ or $2$ as $r=2$ or $r>2$, see~\cite{AFG}).
This time, the constant is known with great accuracy: it is equal to
$2C_2$ where the constant $C_2=0.66016181584686957393$ (see \cite{OEIS}, entry A001692) 
is called the {\em Hardy--Littlewood twin primes constant}.
With a non-monic polynomial $f_2$, it is now more accurate to use Li's improvement 
of the BHC
\begin{eqnarray}\label{eq:Li}
E(x) = C\int_2^x\frac{dt}{\ln(t)\cdot\ln(2t+1)}
\end{eqnarray}
(see Remark~\ref{rem:Li}), as he has shown in~\cite{Li}, where his Table~2 
compares his estimates for $x=10^n$ ($n=2,\ldots, 10$) with those using the original 
BHC formula and with the actual number $Q(x)$. 
We reproduce here his results, removing those of the original BHC,
computing the integrals a little more accurately, and adding the relative errors 
of the estimates, see Table~\ref{tab:Li}.

\begin{table}[htbp]
\begin{tabular}{c|c|c|c}
$x$ & $Q(x)$ & $E(x)$ & relative error \\
\hline
$10^2$    & 10           & 10.20           & 2.00\,\%     \\
$10^3$    & 37           & 39.10           & 5.67\,\%     \\
$10^4$    & 190          & 194.58          & 2.41\,\%     \\
$10^5$    & 1\,171       & 1\,165.95       & $-0.43\,\%$  \\
$10^6$    & 7\,746       & 7\,810.64       & 0.83\,\%     \\
$10^7$    & 56\,032      & 56\,127.94      & 0.17\,\%     \\
$10^8$ 	  & 423\,140     & 423\,294.39     & 0.036\,\%    \\
$10^9$    & 3\,308\,859  & 3\,307\,887.89  & $-0.029\,\%$ \\
$10^{10}$ & 26\,569\,515 & 26\,568\,824.04 & $-0.0026\,\%$ 
\end{tabular}
\smallskip
\caption{$Q(x)$ is the number of $t\le x$ such that both $t$ and $2t+1$ are prime;
$E(x)$ is the estimate of $Q(x)$ given by formula (\ref{eq:Li}).}
\label{tab:Li}
\end{table}

To compare two estimates, we may take the original BHC estimate for $x=10^{10}$, namely,
$$
C \int_2^{10^{10}} \frac{dt}{(\ln t)^2} = 27\,411\,416.53
$$
with the relative error 3.17\,\%. This accuracy is also not bad, but $-0.0026\,\%$ is
significantly better.

The above estimates provide strong support for the conjecture that 
there are infinitely many Sophie Germain primes, and hence that there are infinitely 
many linear groups $G$ in the family under consideration.

Case~(2)(iii) of~\cite[Theorem~1.2]{DZ98} concerns groups $G$ for which $S\cong{\rm PSL}_2(q)$ where the degree $m=(q+1)/2$ is prime and $q=p^{2^k}\ge 5$ for some odd prime $p$ and integer $k\ge 0$. Here we can choose some fixed $k\ge 0$, and take
\begin{equation}\label{eq:PSL2again}
f_1(t)=2t+1\;(=p)\quad\hbox{and}\quad
f_2(t)=\frac{(2t+1)^{2^k}+1}{2}=\sum_{i=1}^{2^k}\binom{2^k}{i}2^{i-1}t^i+1\;(=m).
\end{equation}
For example, if $k=0$ then $f_2(t)=t+1$, so writing $s:=t+1$ we can apply the BHC (+Li) 
to the polynomials $g_1(s)=s\;(=m)$ and $g_2(s)=2s-1\;(=p)$; then $C(g_1,g_2)=1.3203236316\ldots$ again.

\begin{lemm}\label{lem:d=2^k}
Let $f_1$ and $f_2$ be as in\/ {\rm (\ref{eq:PSL2again})}, and denote $d=2^k$, $k\ge 1$. 
Then
\begin{equation}\label{eq:d=2^k}
\omega_f(r) \,=\, \left\lbrace
\begin{array}{ll}
0   & r=2, \\
d+1 & \mbox{\rm if } r\equiv 1\; \mbox{\rm mod } (2d), \\
1	& \mbox{\rm otherwise}.
\end{array}
\right.
\end{equation}
\end{lemm}

\begin{proof}
The equation $f_1=2t+1=0$ has a single root for any $r\ne 2$. What remains is to show that
the equation $f_2(t)=0$ has $d$ roots if $r\equiv 1 \mbox{ mod } (2d)$, and no roots
otherwise.

For each prime $r$, the multiplicative group $U_r$ of units mod $(r)$ is cyclic, of 
order $r-1$. Hence, for any $n$, the number of solutions of $x^n=1$ in $U_r$ is 
$\gcd(n,r-1)$, and the number of elements of order exactly $n$ is $\varphi(n)$ if 
$n$ divides $r-1$ and $0$ otherwise. For $r>2$ the solutions of $x^{2^k}=-1\mbox{ mod }(r)$ 
are the elements of order exactly $2\cdot 2^k=2^{k+1}$, so the number of them is 
$\varphi(2^{k+1})=2^k$ or~$0$ as $2^{k+1}$ divides $r-1$ or not, that is, as 
$r\equiv 1 \mbox{ mod }(2^{k+1})$ or not.
\end{proof}

Lemma \ref{lem:d=2^k} allows us to compute the constants $C(f_1,f_2)$, which we will
denote here by $C(k)$ according to the exponent $k$ in $f_2$, so that $\deg f_2=d=2^k$. 
All the constants in Tables \ref{tab:d=2^k} and \ref{tab:more-const} are computed over 
$r$ up to~$10^9$.

\begin{table}[htbp]
\begin{center}
\begin{tabular}{c|c|c|c|c}
$d=2^k$  & $C(k)$      & $Q(10^9)$   & $E(10^9)$      & relative error  \\
\hline
2        & $4.426783$  & 5\,448\,994 & 5\,448\,648.05 & $-0.006$\,\%    \\
4		 & $10.433814$ & 6\,373\,197 & 6\,365\,668.39 & $-0.118$\,\%    \\
8		 & $7.885346$  & 2\,394\,012 & 2\,395\,075.38 & $0.044$\,\%		\\
16		 & $14.642571$ & 2\,219\,445 & 2\,218\,975.66 & $-0.021$\,\%
\end{tabular}
\end{center}
\smallskip
\caption{$Q(10^9)$ is the number of $t\le 10^9$ such that both $f_1(t)$ and $f_2(t)$
are prime; $E(10^9)$ is the BHC-estimate of $Q(10^9)$.}
\label{tab:d=2^k}
\end{table}

It is too time-consuming to compute further the values of $Q(x)$ since the numbers
$f_2(t)$ become too large, but the computation of the constants $C(k)$ does not present
any additional difficulties. Therefore, we give, in Table~\ref{tab:more-const}, a few 
additional values of this constant.

\begin{table}[htbp]
\begin{center}
\begin{tabular}{c||c|c|c|c|c|c}
$d=2^k$ & 32 		& 64 		& 128 		& 256 		& 512 		& 1024 		\\
\hline
$C(k)$  & 14.424708 & 15.766564 & 12.357306 & 29.736770 & 29.939460 & 32.071863 \\
\hline
$d=2^k$ & 2\,048	& 4\,096	& 8\,192	& 16\,384	& 32\,768	& 65\,536 	\\
\hline
$C(k)$  & 28.880619 & 33.684327 & 33.856467 & 32.037016 & 23.187603 & 44.755201
\end{tabular}
\end{center}
\smallskip
\caption{Constants $C(k)$ for $k=5,\ldots,16$.}
\label{tab:more-const}
\end{table}

%%%%%%%%%

\subsection{Irregular behaviour of the constants $C(k)$}

It is interesting that in this example, as $k$ increases, the Hardy--Littlewood constant $C=C(k)$ also increases, but does not do so monotonically. The following is a heuristic explanation of this curious phenomenon.

For each $k\ge 1$ we have $\omega_f(2)=0$, so~(\ref{eq:BH-C}) gives $C(k)=4\prod_{r>2}c_r$ where
\[c_r=\left(1-\frac{1}{r}\right)^{-2}\left(1-\frac{\omega_f(r)}{r}\right)\]
for each prime $r>2$. Hence
\begin{equation}\label{eq:lnC}
\ln C(k)=\ln 4+\sum_{r>2}\ln c_r,
\end{equation}
where
\[\ln c_r=-2\ln\left(1-\frac{1}{r}\right)+\ln\left(1-\frac{\omega_f(r)}{r}\right)
\approx \frac{2-\omega_f(r)}{r}\]
for each prime $r>2$. Now $\omega_f(r)$ is the number of roots of the polynomial $x(x^{2^k}+1)$ mod~$(r)$, that is, $1+2^k$ or $1$ as $r\equiv 1$ mod~$(2^{k+1})$ or not, so that
\[\ln c_r\approx\frac{1-2^k}{r}\quad\hbox{or}\quad \frac{1}{r}\]
respectively.

Let us define $r_k$ to be the least prime $r\equiv 1$ mod~$(2^{k+1})$, and let us partition the set of primes $r>2$ into three sets: the set $U_k$ of those $r<r_k$, the set $V_k$ of those $r\equiv 1$ mod~$(2^{k+1})$, and the set $W_k$ of those satisfying $r_k\le r\not\equiv 1$ mod~$(2^{k+1})$. Now odd primes $r$ are evenly distributed between the $2^k$ congruence classes of units mod~$(2^{k+1})$, so as $r$ increases, those in $W_k$ appear $2^k-1$ times as frequently as those in $V_k$. It follows that the positive and negative contributions to (\ref{eq:lnC}) of primes in these two sets approximately cancel, leaving just the contributions from primes in $U_k$. Thus
\[\ln C(k)\approx\ln 4+\sum_{r\in U_k}\ln c_r\approx\ln 4+\sum_{2<r<r_k}\frac{1}{r}\approx\ln 4+\ln(\ln r_k)-\frac{1}{2}+b,\]
(see~\cite[Theorem~427]{HW} or~\cite[Theorem 8.8(d)]{NZM}) where
\[b:=\lim_{x\to\infty}\left(\sum_{r<x}\frac{1}{r}-\ln(\ln x)\right)=0.2614972128\ldots\]
is the Meissel--Mertens constant, and hence
\begin{equation}\label{eq:Capprox}
C(k)\approx 4e^{b-1/2}\ln r_k=4e^{b-1/2}\ln(2^{k+1}q_k+1)\approx 4e^{b-1/2}((k+1)\ln 2+\ln q_k)
\end{equation}
where $q_k:=(r_k-1)/2^{k+1}$ for each $k\ge 1$.

Now the sequence of primes
{\small 
$$
r_k = 5, 17, 17, 97, 193, 257, 257, 7681, 12289, 12289, 12289, 40961, 65537, 65537, 65537, 
786433, \ldots
$$
}

\vspace{-5mm}\noindent
gives the sequence of integers $q_k$ shown, for $k=1,\ldots, 16$ in Table~\ref{tab:C(k)q_k},
with the values of $C(k)$, rounded to the nearest integer, shown for comparison.
The irregular behaviour of the terms $\ln q_k$ disturbs the steady increase of the terms $(k+1)\ln 2$ in (\ref{eq:Capprox}).
In particular, if $q_k$ is even then $q_{k+1}=q_k/2$ and hence $\ln q_{k+1}=\ln q_k-\ln 2$, explaining the occasional `plateaux' in the sequence of constants $C(k)$.

\begin{table}[htbp]
\begin{center}
\begin{tabular}{c|c|c|c|c|c|c|c|c|c|c|c|c|c|c|c|c}
$k$ & 1 & 2 & 3 & 4 & 5 & 6 & 7 & 8 & 9 & 10 & 11 & 12 & 13 & 14 & 15 & 16 \\
\hline
$C(k)\approx$  & 4 & 10 & 8 & 15 & 14 & 16 & 12 & 30 & 30 & 32 & 29 & 34 & 34 & 32 & 23 & 45 \\
\hline
$q_k$ & 1 & 2 & 1 & 3 & 3 & 2 & 1 & 15 & 12 & 6 & 3 & 5 & 4 & 2 & 1 & 6	 \\
\end{tabular}
\end{center}
\smallskip
\caption{$C(k)$ (to the nearest integer) and $q_k$ for $k=1,\ldots,16$.}
\label{tab:C(k)q_k}
\end{table}

Both sequences $r_k$ and $b_k$ may be found in the entries A035089 and A035050
of \cite{OEIS}, respectively; to get our sequences, the first two terms of these entries should 
be removed.

The above argument clearly lacks rigour: for example, we have not quantified the errors introduced by the linear approximation of logarithms, or the extent to which the contributions from primes in $V_k$ and $W_k$ `approximately cancel'. Moreover, instead of defining $U_k$ by the inequality $r<r_k$ we could have reduced this upper bound, and argued as before, resulting in a smaller multiplicative constant in (\ref{eq:Capprox}). However, our aim here is explanation rather than precise proof, revealing a cause for the irregular behaviour of the constants $C(k)$ as $k$ increases, rather than trying to estimate them accurately. We leave that to the experts in this area of number theory.

%%%%%%%%%%%

\subsection{Symplectic groups} Case~(3)(i) of~\cite[Theorem~1.2]{DZ98} concerns groups $G$ for which $S\cong{\rm PSp}_{2n}(q)$ where the degree $m=(q^n+1)/2$ is prime, $n\;(>1)$ is a power of $2$, and $q=p^{2^k}$ for some odd prime $p$ and integer $k\ge 0$. Here we can choose some fixed pair $j, k\ge 0$, put $n=2^j$, and take
\begin{equation}\label{eq:PSp}
f_1(t)=2t+1\;(=p)\quad\hbox{and}\quad f_2(t)=\frac{(2t+1)^{2^{j+k}}+1}{2}\;(=m).
\end{equation}
These are the same as the pair $f_1, f_2$ in (\ref{eq:PSL2again}), but with $j+k$ replacing $k$, so the same estimates $E(x)$ and search results $Q(x)$ apply in this case.

%%%%%%%%%%%

\subsection{The remaining families} The other families of groups in~\cite[Theorem~1.2]{DZ98} are either
\begin{itemize}
\item[(a)] obviously infinite, namely
{\begin{itemize}
\item[(a1)] Case~1, with $S\cong{\rm A}_{m+1}$ for primes $m\ge 7$, or
\item[(a2)] Case~2(i), with $S\cong{\rm PSL}_2(m)$ for primes $m\ge 11$, or
\end{itemize}}
\item[(b)] obviously finite, namely Case~5, with
{\begin{itemize}
\item[(b1)] $m=3$ and $S\cong{\rm PSL}_2(9)\cong{\rm A}_6$,
\item[(b2)] $m=7$ with $S\cong{\rm PSp}_6(2)$,
\item[(b3)] $m=11$ with $S\cong{\rm M}_{12}$, or
\item[(b4)] $m=23$ with $S\cong{\rm Co}_2$, ${\rm Co}_3$ or ${\rm M}_{23}$, or
\end{itemize}}
\item[(c)] beyond the scope of the BHC, involving exponential functions rather than polynomials, namely
\begin{itemize}
\item[(c1)] Case~2(iv), with primes $m=2^n-1$ and $S\cong{\rm PSL}_2(2^n)$,
\item[(c1)] Case~2(ii), with primes $m=(3^n-1)/2$ and $S\cong{\rm PSL}_2(3^n)$, or
\item[(c2)] Case~3(ii), with primes $m=(3^n-1)/2$ and $S\cong {\rm PSp}_{2n}(3)$, 
\end{itemize}
where $n$ is an odd prime in all three cases.
\end{itemize}

The primes $m$ appearing in (c1) are the Mersenne primes, while those appearing in (c2) and (c3) appear to be equally difficult to deal with. There is heuristic evidence to support conjectures that both sets are infinite, but proofs seem to be very far away.

%%%%%%%%%%%%%

\subsection{Imprimitive groups}\label{subsec:imprim} An irreducible linear group $G$ of prime degree $m$ is imprimitive if and only if it acts transitively on the $m$ $1$-dimensional subspaces in a direct sum decomposition of its natural module. This action gives an epimorphism from $G$ to a transitive permutation group $H$ of degree $m$; its kernel $D$, conjugate to a group of diagonal matrices, is abelian. We have listed the possibilities for $H$ in Section~\ref{sec:prime}. Now $G$ is solvable if and only if $H$ is, in which case the latter acts as a subgroup of ${\rm AGL}_1(m)$. This case having been dealt with by other authors, Dixon and Zalesskii considered the nonsolvable imprimitive linear groups of degree $m$ in~\cite{DZ04}, this time over an arbitrary algebraically closed field; of course, our results concerning the groups $H={\rm PSL}_n(q)$ have some relevance here. The technical problems to be overcome in classifying the groups $G$ are considerable: given $H$, one has to consider which diagonal groups $D$ it can act on, whether or not the corresponding extensions split, and whether or not the resulting groups $G$ are conjugate in the general linear group. The results obtained in~\cite{DZ04} are too complicated to state here.

In the corrigendum of~\cite{DZ98} it is shown that if $G$ is primitive and the socle $M$ of $G/Z$ is imprimitive and non-abelian, then the commutator subgroup $G'$ is imprimitive and isomorphic to ${\rm PSL}_n(q)$ where $m=(q^n-1)/(q-1)$, with $q$ odd or $G'\cong{\rm PSL}_3(2)$ if $n\ge 3$.  Conversely, such groups $G$ exist provided $m\ge 5$. Again, our results on ${\rm PSL}_n(q)$ are relevant in this case.

%%%%%%%%%%%%%%%%%%

\section{Related problems}

\subsection{Waring's Problem}
Projective primes have occasionally been examined by number theorists, but in a completely different context, that of Waring's Problem (see~\cite[Ch.~XX]{HW}, for example). This asks whether, for each integer $m\ge 1$, there is an integer $g(m)$ such that each positive integer is a sum of at most $g(m)$ $m$-th powers. For instance $g(1)=1$, and $g(2)=4$ by Lagrange's Four Squares Theorem. After Hilbert~\cite{Hil} proved the existence of $g(m)$ in 1909, Tornheim~\cite{Tor} and Bateman and Stemmler~\cite{BaSt} considered similar problems in other number systems. In each case, they took $m$ to be prime for simplicity, and encountered extra difficulties when $m$ was what we have called a projective prime. The reason is that if $m=(q^n-1)/(q-1)$ then every $m$-th power in the field ${\mathbb F}_{q^n}$ lies in the subfield ${\mathbb F}_q$, so elements outside ${\mathbb F}_q$ cannot be sums of $m$-th powers. 

This problem was also important in another way: an investigation in~\cite{BaSt} of the 
frequency of the occurrence of the above phenomenon led Bateman and Horn~\cite{BH} to 
their conjecture.

\subsection{Error-correcting codes} The results on projective primes in Section~\ref{sec:applying} are also relevant to the classification of non-elementary cyclic linear codes of prime length by Guenda and Gulliver in~\cite{GG}, where class~(v) in their Theorem~3 consists of codes with associated permutation group ${\rm P\Gamma L}_n(q)$ acting with prime degree $(q^n-1)/(q-1)$. 

\subsection{Block designs}
A construction by Amarra, Devillers and Praeger \cite{ADP} of block-transitive
point-imprimitive $2$-designs with specific parameters depends on certain polynomials, such as $f(t)=32t^2+20t+1$, taking prime power values. Using all primes $r<10^8$ gives 
$C(f)=4.721240$, and Li's  modified BHC then gives an estimate $E(10^8)=12\,362\,961.06$. In fact, there are $12\,357\,532$ values of $t\le 10^8$ such that $f(t)$ is prime. The relative error is $0.044\%$. See~\cite{JZ:block} for other polynomials and further details.

\subsection{Difference sets}
A construction of divisible difference sets by Fern\'andez-Alcober, Kwashira and Mart\'{\i}nez in~\cite[Proposition~4.1]{FKM} depends on the existence of prime powers $q$ such that $3q-2$ is also a prime power. In such generality, this situation is beyond the scope of the BHC, but one can deal with the case where $q$ and $3q-2$ are both prime by applying it to the polynomials $f_1=t$ and $f_2=3t-2$. We find that $C=2.640647$, leading to an estimate $E(10^9)=6\,485\,752.27$ for the number of such primes $q\le 10^9$, compared with the actual number $Q(10^9)=6\,484\,218$ found by computer search. The relative error is $0.024\%$.

More generally, one can deal with the case where $q$ is a prime power and $3q-2$ is prime by taking $f_1=t$ and $f_2=3t^e-2$ for some fixed $e$. For example, if we take $e=2$ then $C=2.540480$, giving an estimate $E(10^9)=3\,205\,208.84$ for the number of such primes $t=\sqrt{q}\le 10^9$, compared with the actual number $Q(10^9)=3\,203\,900$. In this case the relative error is $0.041\%$.

One can also deal with the case where $q$ is prime and $3q-2$ is a proper prime power $p^e$ by using the polynomials $t\;(=p)$ and $(t^e+2)/3\;(=q)$ for fixed $e\ge 2$; we need $t^e\equiv 1$ mod~$(3)$ in order to have integer coefficients, so write $t=3s+1$, giving polynomials $3s+1$ and $((3s+1)^e+2)/3$ in ${\mathbb Z}[s]$, or $t=3s-1$ with $e$ even, giving polynomials $3s-1$ and $((3s-1)^e+2)/3$ in ${\mathbb Z}[s]$. For instance if $e=2$ we can apply the BHC to two pairs of polynomials, namely $3s+1, 3s^2+2s+1$, and $3s-1, 3s^2-2s+1$;  in the first case we get an estimate of $4\,892\,910.99$ and an actual number $4\,893\,804$ (error $-0.018\%$), with corresponding values $4\,892\,911.60$ and $4\,894\,315$ (error $-0.029\%$) in the second case.

These results strongly suggest that there are infinitely many pairs of prime powers $q$ and $3q-2$ with at least one of them prime. This might suggest a similar conjecture in the remaining case, where both $q$ and $3q-2$ are proper prime powers, but here we have Mihailescu's proof of the Catalan Conjecture as a warning. Similarly, there is Pillai's conjecture that for fixed integers $A, B, C>0$ there are only finitely many integer solutions of the equation $Ax^m-By^n=C$ with $m, n>2$. Thus the status of this part of the construction seems to be an interesting open problem.

\medskip

We close with two conjectures which, while much less important than those considered in Section~\ref{sec:conj}, nevertheless have their own interest.

\subsection{The Goormaghtigh Conjecture (1917)}

Since
\[1+2+2^2+2^3+2^4=31=1+5+5^2,\]
both ${\rm PSL}_5(2)$ and ${\rm PSL}_3(5)$ have natural degree $31$.
Goormaghtigh, a Belgian engineer and amateur mathematician, conjectured in~\cite{Goo} that this example and
\[1+2+2^2+\cdots+2^{12}=8191=1+90+90^2\]
are the only positive integer solutions of
$(x^n-1)/(x-1)=(y^k-1)/(y-1)$
with $n\ne k$ and $n, k\ge 3$. This conjecture is still open.
Although $8191$ is prime, $90$ is not a prime power,
so only the first example is relevant to permutation groups ${\rm PSL}_n(q)$.

%%%%%%%%%%%

\subsection{The Feit--Thompson Conjecture (1962)}

Feit and Thompson~\cite{FT62} conjectured that if $p$ and $q$ are distinct primes then
$(p^q-1)/(p-1)$ does not divide $(q^p-1)/(q-1)$.
They stated that if true this would significantly shorten their 255-page proof~\cite{FT63} that groups of odd order are solvable.
However, an alternative simplification was found by Peterfalvi~\cite{Pet} in 1984.
The conjecture has been proved by Le~\cite{Le12} for $q=3$, but it is otherwise still open.

%%%%%%%%%%%%%%%%%%%

\section{Acknowledgements}

We are greatly indebted to Yuri Bilu who acquainted us with the Bunyakovsky conjecture, 
which became a crucial point of our study, and to Peter Cameron, Robert Guralnick, Weixiong Li, Valery 
Liskovets, Cheryl Praeger and Dimitri Zvonkine for some very useful comments.
Jean B\'etr\'ema helped us with some computations which were too onerous for our Maple
package on a laptop computer. 
Alexander Zvonkin is  partially supported by the ANR project {\sc Combin\'e} (ANR-19-CE48-0011).

%%%%%%%%%%%%%%%%%%%

\end{document}